\documentclass{amsart}

\newtheorem{theorem}{Theorem}[section]
\newtheorem{lemma}[theorem]{Lemma}
\newtheorem{cor}[theorem]{Corollary}
\theoremstyle{definition}
\newtheorem{definition}[theorem]{Definition}
\newtheorem{example}[theorem]{Example}

\theoremstyle{remark}
\newtheorem{remark}[theorem]{Remark}

\numberwithin{equation}{section}

\begin{document}
\title[Biwarped product submanifolds]{Biwarped product submanifolds\\
of a K\"{a}hler manifold}

\author[H.M. Ta\c stan]{Hakan  Mete Ta\c stan}

\address{Department of Mathematics\\
Faculty of Science\\
\.Istanbul University\\
Vezneciler, \.Istanbul, Turkey}
\email{hakmete@istanbul.edu.tr}


\begin{abstract}
We study biwarped product submanifolds which are special cases of multiply warped product submanifolds in K\"{a}hler manifolds.
We observe the non-existence of such submanifolds under some circumstances.
We show that there exists a non-trivial biwarped product submanifold of a certain type by giving an illustrate example.
We also give a necessary and sufficient condition for such submanifolds to be locally trivial.
Moreover, we establish an inequality for the squared norm of the second fundamental form in terms
of the warping functions for such submanifolds. The equality case is also discussed.
\end{abstract}
\maketitle
\section{Introduction}
Bishop and O' Neill \cite{Bi} introduced the concept of warped product of Riemannian manifolds to
construct a large class of complete manifolds of negative curvature.
This concept is also a generalization of the usual product of Riemannian manifolds. N\"{o}lker \cite{No} considered the notion of
multiply warped products as a generalization of the  warped products. Since that time, multiply warped products has been studied by many authors.
For example, \"{U}nal \cite{U} studied partially the geometry of the multiply warped products when the metrics of such products are Lorentzian.
Curvature properties of such products were investigated by Dobarro and \"{U}nal \cite{Do}.\\

The concept of warped products or multiply warped products play very important roles in physics as well as in differential geometry,
especially in the theory of relativity. Actually, the standard spacetime models such as Robertson-Walker, Schwarschild, static and Kruscal are
warped products. Also, the simplest models of neighborhoods of stars and black holes are warped products \cite{O}. Moreover, many solutions to
Einstein's field equation can be expressed in terms of warped products \cite{Bee}.\\

In differential geometry, especially in almost complex geometry, one of the most intensively research areas
is the theory of submanifolds. Actually, the almost complex structure of an almost Hermitian manifold
determines several classes of submanifolds such as holomorphic(invariant), totally real(anti-invariant) \cite{Y},
CR- \cite{Be}, generic \cite{Che0}, slant \cite{Che1}, semi-slant \cite{Pa}, hemi-slant(pseudo-slant) \cite{Ca, S3}, pointwise slant \cite{Che4,Et},
bi-slant \cite{Ca}, skew CR- and generic submanifolds \cite{Ro}. Among them, the last one contains all other classes.\\

The theory of warped product submanifolds has been becoming a popular research area since Chen \cite{Che2} studied
the warped product CR-submanifolds in K\"{a}hler manifolds. Actually, several classes of warped product submanifolds appeared
in the last fifteen years (see \cite{Sr,S2,S3,S4,S5}). Also, warped product submanifolds have been studying in different kinds of structures
such as nearly K\"{a}hler \cite{S7}, para-K\"{a}hler \cite{Che6}, locally product Riemannian \cite{S6,Ta},
cosymplectic \cite{Ud}, Sasakian \cite{Mi}, generalized Sasakian \cite{Su}, trans-Sasakian \cite{Kh1}, $(\kappa,\mu)-$ \cite{Tr},
Kenmotsu \cite{Kh,Mu} and quaternion \cite{Ma}. Most of the studies related to the theory of
warped product submanifolds can be found in Chen's coming book \cite{Che3}. Recently, Chen and Dillen \cite{Che5}
studied multiply warped product submanifolds in K\"{a}hler manifolds and they obtained very useful optimal inequalities.
We note that such submanifolds were also studied in Kenmotsu manifolds \cite{Ol}.\\

In this paper, we consider and study biwarped product submanifolds in K\"{a}hler manifolds.
Here, a biwarped product means that a multiply warped product which has only two fibers.
We observe the non-existence of biwarped product submanifolds under some circumstances.
After giving an illustrate example, we study such submanifolds in case of the base factor is
holomorphic and one of the two fibers is totally real
and the other one is pointwise slant submanifold.
We also give characterization for this kind of submanifolds.
Moreover, we investigate the behavior of the second fundamental form of such submanifolds and as a result,
we give a necessary and sufficient condition for such manifolds to be locally trivial.
Furthermore, an inequality for the squared norm of the second fundamental form in terms
of the warping functions for such submanifolds is obtained. The equality case is also
considered.
\section{Preliminaries}
In this section, we recall the fundamental definitions and notions needed further study. Actually, in subsection 2.1,
we will recall the definition of the multiply warped product manifolds. In subsection 2.2, we will give the basic background for submanifolds of Riemannian manifolds.
The definition of a K\"{a}hler manifold and the  some classes of submanifolds of K\"{a}hler manifolds   are placed in subsection 2.3.
\subsection{Multiply product manifolds}
Let $(M_{i},g_{i})$ be Riemannian manifolds for any $i\in\{0,1,...,k\}$ and let $f_{j}:M_{0}\rightarrow (0,\infty)$  be smooth functions for any
$j\in\{1,2,...,k\}.$ Then the \emph{multiply product manifold} \cite{No} $\bar{M}=M_{0}\times_{f_{1}}M_{1}\times...\times_{f_{k}}M_{k}$
is the product manifold $\tilde{M}=M_{0}\times M_{1}\times...\times M_{k}$ endowed with the metric
$$g=\pi^{*}_{0}(g_{0})\oplus (f_{1}\circ\pi_{0})^{2}\pi^{*}_{1}(g_{1})\oplus...\oplus (f_{k}\circ\pi_{0})^{2}\pi^{*}_{k}(g_{k}).$$
More precisely, for any vector fields $X$ and $Y$ of $\bar{M}$, we have
$$g(X,Y)=g_{0}(\pi_{0*}X,\pi_{0*}Y)+\displaystyle\sum^{k}_{i=1}(f_{i}\circ\pi_{0})^{2}g_{i}(\pi_{i*}X,\pi_{i*}Y),$$
where $\pi_{i}:\tilde{M}\rightarrow M_{i}$ is the canonical projection of $\tilde{M}$ onto $M_{i}$, \,
$\pi^{*}_{i}(g_{i})$ is the pullback of $g_{i}$ by $\pi_{i}$  and the subscript $*$
denotes the derivative map of $\pi_{i}$ for each $i.$ Each function $f_{j}$ is called a \emph{warping function} and each manifold $(M_{j},g_{j})$, $j\in\{1,2,...,k\}$ is called a \emph{fiber} of the multiply warped product $\bar{M}.$ The manifold $(M_{0},g_{0})$ is called a \emph{base manifold} of $\bar{M}.$ As well known, the base manifold of $\bar{M}$ is totally geodesic and the fibers of $\bar{M}$ are totally umbilic in $\bar{M}$.\\

Let $\bar{M}=M_{0}\times_{f_{1}}M_{1}\times...\times_{f_{k}}M_{k}$ be multiply product manifold,
if $k=1$, then we get a (singly) warped product \cite{Bi}. We call the multiply product manifolds as \emph{biwarped product manifolds} for $k=2.$
In other words, a biwarped product manifold has the form $M_{0}\times_{f_{1}}M_{1}\times_{f_{2}}M_{2}.$ We say that
a biwarped product manifold is \emph{trivial}, if the warping functions  $f_{1}$ and $f_{2}$ are constants. Note that
biwarped product manifolds were also studied under the name of \emph{twice warped products} \cite{Ba}.\\

Let $\bar{M}=M_{0}\times_{f_{1}}M_{1}\times_{f_{2}}M_{2}$ be a biwarped product manifold with the Levi-Civita connection $\bar{\nabla}$
and $^{i}\nabla$ denote the Levi-Civita connection of $M_{i}$ for $i\in\{0,1,2\}.$ By usual convenience, we denote the set of lifts of vector fields on $M_{i}$ by $\mathcal{L}(M_{i})$ and use the same notation for a vector field and for its lift. On the other hand, since the map $\pi_{0}$ is an isometry and $\pi_{1}$ and $\pi_{2}$ are (positive) homotheties, they preserve the Levi-Civita connections. Thus, there is no confusion using the same notation for a connection on $M_{i}$ and for its pullback via $\pi_{i}.$ Then, the covariant derivative formulas for a biwarped product manifold are given by the following.
\begin{lemma} \label{lem001}(\cite{Ba}) Let $\bar{M}=M_{0}\times_{f_{1}}M_{1}\times_{f_{2}}M_{2}$ be a biwarped product manifold. Then, we have
\begin{align}
&\bar{\nabla}_{U}V=\,^{0}\nabla_{U}V, \label{a1}\\
&\bar{\nabla}_{V}X=\bar{\nabla}_{X}V=V(\ln f_{i})X,\label{a2}\\
&\bar{\nabla}_{X}Z= \begin{cases}
            0         &  \mbox{if } i\neq j, \\
            ^{i}\nabla_{X}Z-g(X,Z)\emph{grad}\,(\ln f_{i}) &  \mbox{if } i=j,
       \end{cases}\label{a3}\qquad\qquad\qquad
\end{align}
where $U,V\in\mathcal{L}(M_{0})$, $X\in\mathcal{L}(M_{i})$  and  $Z\in\mathcal{L}(M_{j})$.
\end{lemma}
We note that grad$(\ln f_{i})\in\mathcal{L}(M_{0})$ for $i=1,2$ \cite{No}.
\subsection{Submanifolds of Riemannian manifolds}
Let $M$ be an isometrically immersed submanifold in a Riemannian manifold $(\bar{M},g)$.
Let $\bar{\nabla}$ is the Levi-Civita connection of $\bar{M}$ with respect to the metric $g$
and let $\nabla$ and $\nabla^{\bot}$ be the induced, and induced normal connection on $M$, respectively. Then, for all  $U,V\in TM$ and
$\xi\in T^{\bot}M$, the Gauss and Weingarten formulas are given respectively by
\begin{equation}
\label{e2}
\begin{array}{c}
\bar{\nabla}_{U}V=\nabla_{U}V+h(U,V)
\end{array}
\end{equation}
and
\begin{equation}
\label{e3}
\begin{array}{c}
\bar\nabla_{U}\xi=-A_{\xi}U+\nabla_{U}^{\bot}\xi
\end{array},
\end{equation}
where $TM$ is the tangent bundle and $T^{\bot}M$ is the normal bundle of $M$ in $\bar{M}$.
Additionally, $h$ is the \emph{second fundamental form} of $M$ and $A_{\xi}$
is the Weingarten endomorphism associated with $\xi.$ The second
fundamental form $h$ and the \emph{shape operator} $A$ related by
\begin{equation}
\label{e4}
\begin{array}{c}
g(h(U,V),\xi)=g(A_{\xi}U,V)
\end{array}.
\end{equation}
The \emph{mean curvature vector field} $H$ of $M$ is given by
$H=\frac{1}{m}(trace\,h),$ where $dim(M)=m.$ We say that the submanifold $M$ is
\emph{totally geodesic} in $\bar{M}$ if $h=0$, and
\emph{minimal} if $H=0.$ The submanifold $M$ is called \emph{totally umbilical}  if $h(U,V)=g(U,V)H$ for all $U, V\in TM.$
If the manifold $M$ is totally umbilical and its mean curvature vector field $H$ is \emph{parallel}, i.e. $g(\bar\nabla_{U}H,\xi)=0$ for all
$U\in TM$ and $\xi\in T^{\bot}M$, then the submanifold $M$ is said to be \emph{spherical} or \emph{extrinsic sphere}.\\

Let $\mathcal{D}^{1}$ and $\mathcal{D}^{2}$ be any two distributions on $M$. Then we say that
$M$ is $\mathcal{D}^{1}$-\emph{geodesic}, if $h(U,V)=0$ for all $U, V\in \mathcal{D}^{1}$ and we say that
$(\mathcal{D}^{1},\mathcal{D}^{2})$-\emph{mixed geodesic} if $h(V,X)=0$ for $V\in \mathcal{D}^{1}$ and $X\in \mathcal{D}^{2}.$
If for all $V\in \mathcal{D}^{1}$ and $X\in \mathcal{D}^{2}$,  $\nabla_{X}V\in\mathcal{D}^{1}$, then $\mathcal{D}^{1}$ is called $\mathcal{D}^{2}$-\emph{parallel}.
We say that $\mathcal{D}^{1}$ is \emph{autoparallel} if $\mathcal{D}^{1}$ is $\mathcal{D}^{1}$-parallel. If a distribution on $M$ is autoparallel,
then by  the Gauss formula it is totally geodesic.
\subsection{Some classes of submanifolds of K\"{a}hler manifolds}
Let $\bar{M}$ be an almost complex manifold with almost complex structure $J.$ If there is a Riemannian metric $g$ on $\bar{M}$ satisfying
\begin{equation}
\label{e5}
\begin{array}{c}
g(JX,JY)=g(X,Y)
\end{array}
\end{equation}
for any $X,Y\in T\bar{M},$ then we say that $(\bar{M},J,g)$ is an almost Hermitian manifold.
Let $\bar{\nabla}$ be the Levi-Civita connection of the almost Hermitian manifold $(\bar{M},J,g)$ with respect to $g.$
Then $(\bar{M},J,g)$ is called a \emph{K\"{a}hler manifold} \cite{Y} if $J$ is parallel with respect to $\bar{\nabla},$
i.e.,
\begin{equation}
\label{e6}
\begin{array}{c}
(\bar\nabla_{X}J)Y=0
\end{array}
\end{equation}
for all $X,Y\in T\bar{M}.$\\

Let $M$ be a Riemannian manifold isometrically immersed submanifold in a K\"{a}hler  manifold $(\bar{M},J,g)$. Then the submanifold
$M$ is called a \emph{pointwise slant submanifold} \cite{Che4, Et} if for every point $p$ of $M,$ the Wirtinger angle $\theta(V)$ between $JV$
and the tangent space $T_{p}M$ at $p$ is independent of the choice of the nonzero vector $V\in T_{p}M.$ In this case, the angle $\theta$
can be viewed as a function on $M$ and it is called the \emph{slant function} of $M$. We say that the pointwise slant submanifold $M$ is
\emph{proper} neither $\cos\theta(p)=0$ nor $\sin\theta(p)=0$ at each point $p\in M$. (This condition is different from Chen's definition, see \cite{Che4}).\\

Now, let $M$ be a submanifold of a K\"{a}hler  manifold $(\bar{M},J,g)$. For any $V\in TM$, we put
\begin{equation}
\label{e7}
\begin{array}{c}
JV=TV+FV
\end{array}.
\end{equation}
Here $TV$ is the tangential part of $JV,$ and $FV$ is the normal
part of $JV.$ Then $M$ is  a pointwise slant submanifold of $\bar{M}$ if and only if, for any $V\in TM$, we have
\begin{equation}
\label{e8}
\begin{array}{c}
T^{2}V=-\cos^{2}\!\theta V
\end{array}
\end{equation}
for some function $\theta$ defined on $M$ \cite{Che4}. For a pointwise slant submanifold of $\bar{M}$,
using (\ref{e7}), (\ref{e8}) and the K\"{a}hler structure, it is not difficult to prove the following two facts.
\begin{equation}
\label{e9}
\begin{array}{c}
g(TU,TV)=\cos^{2}\!\theta g(U,V)
\end{array},
\end{equation}
\begin{equation}
\label{e10}
\begin{array}{c}
g(FU,FV)=\sin^{2}\!\theta g(U,V)
\end{array}
\end{equation}
for $U,V\in TM$.\\

Let $M$ be a pointwise slant submanifold with slant function $\theta$ of a K\"{a}hler  manifold $(\bar{M},J,g).$
If the function $\theta$ is a constant, i.e., it is also independent of the choice of the point $p\in M$, then we say that
$M$ is a \emph{slant submanifold} \cite{Che1}.\\

If $\theta\equiv0,$ then $M$ is called a \emph{holomorphic} or \emph{complex submanifold} \cite{Y}. In that case, the tangent space
$T_{p}M$ is invariant with respect to the almost complex structure $J$ at each point $p\in M$, i.e., $J(T_{p}M)\subseteq T_{p}M.$\\

If $\theta\equiv\frac{\pi}{2},$ then $M$ is called a  \emph{totally real submanifold} of   \cite{Y}. In which case, the tangent space
$T_{p}M$ is anti-invariant with respect to the almost complex structure $J$ for every point $p$ of $M$, i.e., $J(T_{p}M)\subseteq T^{\bot}_{p}M.$
\section{Generalized $J$-induced submanifolds of order 1}
In this section, after renaming the generic submanifolds (in the sense of Ronsse \cite{Ro}), we will give some results concerning
totally geodesicness and integrability of the distributions which are involved in the definition of such submanifolds.\\

The most general class of submanifolds determined by the almost complex structure is the class of
generic submanifols which was defined by Ronsse \cite{Ro}. There are two other classes of submanifolds with the same name.
One of these is the class of defined by Chen \cite{Che0} and the other one is the class defined by Yano and Kon \cite{Y0}.
Because of these facts, to avoid name confusion, we call the generic submanifolds (in the sense of Ronsse \cite{Ro}) as
\emph{generalized structure induced} or \emph{generalized J-induced submanifolds}.
\begin{definition} (\cite{Ro}) Let $M$ be a submanifold of a K\"{a}hler manifold $(\bar{M},J,g).$
Then $M$ is said to be a\emph{ generalized structure induced} or \emph{generalized J-induced submanifold} if the tangent bundle $TM$ of $M$
has the form
\begin{equation}\nonumber
\begin{array}{c}
TM=\mathcal{D}^T\oplus\mathcal{D}^\bot\oplus\mathcal{D}^{\theta_{1}}\oplus...\oplus\mathcal{D}^{\theta_{k}}
\end{array},
\end{equation}
where $\mathcal{D}^T$ is a holomorphic, $\mathcal{D}^\bot$ is a totally real and each of $\mathcal{D}^{\theta_{i}}$ is
a pointwise slant distribution on $M$ and $\theta_{i}$'s are distinct for $i=1,...,k.$ In addition, if each of $\mathcal{D}^{\theta_{i}}$
is a slant distribution, then we say that $M$ is a \emph{skew CR-submanifold} of $\bar{M}.$
\end{definition}
In a special case, we have the following definition.
\begin{definition} A submanifold $M$ of a K\"{a}hler  manifold $(\bar{M},J,g)$ is called
a\emph{ generalized structure induced submanifold of order 1} or \emph{generalized $J$-induced submanifold of order 1} if it
is a generalized $J$-induced submanifold with $k=1.$
\end{definition}
In which case, we have
\begin{equation}\label{g1}
\begin{array}{c}
TM=\mathcal{D}^T\oplus\mathcal{D}^\bot\oplus\mathcal{D}^\theta
\end{array}.
\end{equation}
Thus, the normal bundle $T^{\bot}M$ of
$M$ is decomposed as
\begin{equation}\label{g2}
\begin{array}{c}
T^{\bot}M=J(\mathcal{D}^{\bot})\oplus F(\mathcal{D}^{\theta})\oplus\overline{\mathcal{D}}^T
\end{array},
\end{equation}
where $\overline{\mathcal{D}}^T$ is the orthogonal complementary distribution of
$J(\mathcal{D}^{\bot})\oplus F(\mathcal{D}^{\theta})$ in $T^{\bot}M$
and it is an invariant subbundle of $T^{\bot}M$ with respect to $J.$\\

We say that a generalized $J$-induced submanifold of order 1 is \emph{proper}, if $\mathcal{D}^T\neq\{0\}$,
$\mathcal{D}^{\bot}\neq\{0\}$ and the slant function $\theta$ belongs to open interval $(0,\frac{\pi}{2}).$\\

For the further study of generalized $J$-induced submanifolds of order 1 of a K\"{a}hler  manifold, we need the following lemma.
\begin{lemma} \label{lem01} Let $M$ be a generalized $J$-induced submanifold of order 1 of a K\"{a}hler  manifold $(\bar{M},J,g)$. Then, we have
\begin{equation}
\label{g3}
\begin{array}{c}
g(\nabla_{U}V,Z)=\csc^{2}\!\theta g(A_{FZ}JV-A_{FTZ}V,U)
\end{array},
\end{equation}
\begin{equation}
\label{g4}
\begin{array}{c}
g(\nabla_{Z}W,V)=\csc^{2}\!\theta g(A_{FTW}V-A_{FW}JV,Z)
\end{array},
\end{equation}
where $U,V\in\mathcal{D}^T$ and  $Z,W\in\mathcal{D}^\theta$.
\end{lemma}
\begin{proof} Using (\ref{e2}), (\ref{e5}) and (\ref{e6}), we have
\begin{align*}
g(\nabla_{U}V,Z)&=g(\bar{\nabla}_{U}J,JZ)=g(\bar{\nabla}_{U}JV,TZ)+g(\bar{\nabla}_{U}JV,FZ)\\
&=-g(\bar{\nabla}_{U}V,T^{2}Z)-g(\bar{\nabla}_{U}V,FTZ)+g(\bar{\nabla}_{U}JV,FZ).
\end{align*}
for $U,V\in\mathcal{D}^T$ and  $Z\in\mathcal{D}^\theta$. Now, using (\ref{e8}) and (\ref{e3}), we obtain
\begin{equation}\nonumber
\begin{array}{c}
\sin^{2}\!\theta g(\nabla_{U}V,Z)=g(A_{FZ}JV-A_{FTZ}V,U)
\end{array}.
\end{equation}
This gives (\ref{g3}). The other assertion can be obtained in a similar way.
\end{proof}
\begin{lemma} \label{lem02} Let $M$ be a generalized $J$-induced submanifold of order 1 of a K\"{a}hler  manifold $(\bar{M},J,g)$. Then, we have
\begin{align}
&g(\nabla_{U}V,X)=g(A_{JX}U,JV),\label{g5}\\
&g(\nabla_{U}X,Z)=-\sec^{2}\!\theta g(A_{JX}TZ-A_{FTZ}X,U),\label{g6}\\
&g(\nabla_{X}Y,U)=-g(A_{JY}X,JU),\label{g7}\\
&g(\nabla_{Z}W,Y)=\sec^{2}\!\theta g(A_{JY}TW-A_{FTW}Y,Z),\label{g8}\\
&g(\nabla_{Z}Y,U)=-g(A_{JY}JU,Z),\label{g9}\\
&g(\nabla_{X}Y,Z)=-\sec^{2}\!\theta \{g(A_{JY}X,TZ)-g(A_{FTZ}X,Y)\},\label{g10}\\
&g(\nabla_{X}U,Z)=\csc^{2}\!\theta \{g(A_{FZ}JU-A_{FTZ}U,X)\},\label{g11}
\end{align}
where $U,V\in\mathcal{D}^{T}$, $X,Y\in\mathcal{D}^\bot$ and  $Z,W\in\mathcal{D}^\theta$.
\end{lemma}
\begin{proof} The proofs of all equations are same as the proofs of equations of Lemma 1, Lemma 2 and Lemma 3 of \cite{S4}.
So, we omit them.
\end{proof}
As applications of Lemma (\ref{lem01}) and Lemma (\ref{lem02}), we have the following results.
\begin{theorem} \label{thm01} Let $M$ be a generalized $J$-induced submanifold of order 1 of a K\"{a}hler  manifold $(\bar{M},J,g)$. Then
the holomorphic distribution $\mathcal{D}^{T}$ is totally geodesic if and only if the following equations hold
\begin{equation}\label{g12}
\begin{array}{c}
g(A_{JX}U,JV)=0
\end{array},
\end{equation}
\begin{equation}\label{g13}
\begin{array}{c}
g(A_{FZ}U,JV)=g(A_{FTZ}U,V)
\end{array}
\end{equation}
for $U,V\in\mathcal{D}^{T}$, $X\in\mathcal{D}^\bot$ and  $Z\in\mathcal{D}^\theta$.
\end{theorem}
\begin{proof}  Let $M$ be a generalized $J$-induced submanifold of order 1 of a K\"{a}hler  manifold $(\bar{M},J,g)$.
Then, the holomorphic distribution $\mathcal{D}^{T}$ is totally geodesic if and only if
$g(\nabla_{U}V,X)=0$ and $g(\nabla_{U}V,Z)=0$ for all $U,V\in\mathcal{D}^{T}$, $X\in\mathcal{D}^\bot$ and  $Z\in\mathcal{D}^\theta$.
Thus, both assertions follow from (\ref{g5}) and (\ref{g3}), respectively.
\end{proof}
\begin{theorem} \label{thm02} Let $M$ be a generalized $J$-induced submanifold of order 1 of a K\"{a}hler  manifold $(\bar{M},J,g)$. Then
the pointwise slant distribution $\mathcal{D}^\theta$ is integrable if and only if the following equations hold
\begin{equation}\label{g14}
\begin{array}{c}
g(A_{FTZ}V-A_{FZ}JV,W)=g(A_{FTW}V-A_{FW}JV,Z)
\end{array},
\end{equation}
\begin{equation}\label{g15}
\begin{array}{c}
g(A_{JX}TW-A_{FTW}X,Z)=g(A_{JX}TZ-A_{FTZ}X,W)
\end{array}
\end{equation}
for $V\in\mathcal{D}^{T}$, $X\in\mathcal{D}^\bot$ and  $Z,W\in\mathcal{D}^\theta$.
\end{theorem}
\begin{proof} Let $M$ be a generalized $J$-induced submanifold of order 1 of a K\"{a}hler  manifold $(\bar{M},J,g)$.
Then, the pointwise slant  distribution $\mathcal{D}^{\theta}$ is integrable if and only if
$g([Z,W],V)=0$ and $g([Z,W],X)=0$ for all $U\in\mathcal{D}^{T}$, $X\in\mathcal{D}^\bot$ and  $Z,W\in\mathcal{D}^\theta$.
Thus, both assertions follow from (\ref{g4}) and (\ref{g8}), respectively.
\end{proof}
\begin{remark} \label{rema0} From Lemma 1 of \cite{Ro} or Theorem 4.2 of \cite{Tr0}, we know that the totally real distribution
$\mathcal{D}^\bot$ is always integrable.
\end{remark}
\section{Biwarped product submanifolds of K\"{a}hler manifolds}
In this section, we check that the existence of biwarped product submanifolds in the
form $M_{0}\times_{f_{1}}M_{1}\times_{f_{2}}M_{2}$ of a K\"{a}hler  manifold $(\bar{M},J,g)$,
where $M_{0},M_{1}$ and $M_{2}$ are one of the submanifolds given in subsection 2.3.
\subsection{Existence problems}
Chen proved that there do not exist (non-trivial) warped product submanifolds in the form
$M_{\bot}\times_{f}M_{T}$ in a K\"{a}hler  manifold $\bar{M}$ such that $M_{\bot}$ is a totally real and $M_{T}$ is a holomorphic submanifold of $\bar{M}$
in Theorem 3.1 of \cite{Che1}, Thus, we obtain the following result.
\begin{cor} There do not exist (non-trivial) biwarped product submanifolds of type $M_{\bot}\times_{f}M_{T}\times_{\sigma}M_{\theta}$
of a K\"{a}hler  manifold $(\bar{M},J,g)$ such that $M_{\bot}$ is a totally real, $M_{T}$ is a holomorphic and $M_{\theta}$ is
a pointwise slant or slant submanifold of $\bar{M}$.
\end{cor}
In Theorem 3.1 of \cite{S3}, \c{S}ahin showed that there is no (non-trivial) warped product submanifolds of the form $M_{\phi}\times_{f}M_{T}$ in a K\"{a}hler manifold $\bar{M}$ such that $M_{\phi}$ is a proper slant and $M_{T}$ is a holomorphic submanifold of $\bar{M}$. In Theorem 4.1 of \cite{S4}, he also proved  that
there exist no (non-trivial) warped product submanifolds of the form $M_{\theta}\times_{f}M_{T}$ in a K\"{a}hler  manifold $\bar{M}$ such that $M_{\theta}$ is a proper pointwise slant and $M_{T}$ is a holomorphic submanifold of $\bar{M}.$ Hence, we conclude that:
\begin{cor} Let $\bar{M}$ be a K\"{a}hler  manifold. Then there exist no (non-trivial) biwarped product submanifolds in the form $M_{\theta}\times_{f}M_{T}\times_{\sigma}M_{\bot}$ of $\bar{M}$ such that $M_{\theta}$ is a proper pointwise slant or proper slant,
$M_{T}$ is a holomorphic and $M_{\bot}$ is a totally real submanifold of $\bar{M}$.
\end{cor}
On the other hand, in Theorem 3.2 of \cite{S3}, \c{S}ahin proved that there exists no (non-trivial)
warped product submanifolds of the form $M_{T}\times_{f}M_{\phi}$ in a K\"{a}hler  manifold $\bar{M}$
such that $M_{T}$ is a holomorphic and $M_{\phi}$ is a proper slant submanifold of $\bar{M}.$ Thus, we get the following result.
\begin{cor} There exist no (non-trivial) biwarped product submanifolds of type $M_{T}\times_{f}M_{\phi}\times_{\sigma}M_{\bot}$
of a K\"{a}hler  manifold $(\bar{M},J,g)$ such that $M_{T}$ is a holomorphic, $M_{\phi}$ is
a proper slant and  $M_{\bot}$ is a totally real  submanifold of $\bar{M}$.
\end{cor}
Now, we consider (non-trivial) biwarped product submanifolds of the form $M_{T}\times_{f}M_{\bot}\times_{\sigma}M_{\theta}$
in a K\"{a}hler  manifold $(\bar{M},J,g)$ such that $M_{T}$ is a holomorphic, $M_{\bot}$ is a totally real and $M_{\phi}$ is
a pointwise slant submanifold of $\bar{M}$. We first present an example of such a submanifold.
\begin{example} Consider the K\"{a}hler  manifold $\mathbb{R}^{14}$ with usual K\"{a}hler structure.
For $u,v\neq0,1$ and $x,z,w\in(0,\frac{\pi}{2})$ we consider a submanifold $M$ in $\mathbb{R}^{14}$ given by
\begin{align*}
&y_{1}=u\cos\! z,\, y_{2}=v\cos\! z,\, y_{3}=u\cos\! w,\, y_{4}=v\cos\! w ,\\
&y_{5}=u\sin\! z,\, y_{6}=v\sin\! z,\, y_{7}=u\sin\! w,\,y_{8}=v\sin\! w ,\\
&y_{9}=z, \,y_{10}=w,\, y_{11}=u\cos\! x,\, y_{12}=v\cos\! x,\, y_{13}=u\sin\! x,\, y_{14}=v\sin\! x.
\end{align*}
Then, we see that the local frame of the tangent bundle $TM$ of $M$ is spanned by
\begin{align*}
U&=\cos\! z\frac{\partial}{\partial y_{1}}+\cos\! w\frac{\partial}{\partial y_{3}}+\sin\! z\frac{\partial}{\partial y_{5}}+\sin\! w\frac{\partial}{\partial y_{7}}+\cos\! x\frac{\partial}{\partial y_{11}}+\sin\! x\frac{\partial}{\partial y_{13}},\\
V&=\cos\! z\frac{\partial}{\partial y_{2}}+\cos\! w\frac{\partial}{\partial y_{4}}+\sin\! z\frac{\partial}{\partial y_{6}}+\sin\! w\frac{\partial}{\partial y_{8}}+\cos\! x\frac{\partial}{\partial y_{12}}+\sin\! x\frac{\partial}{\partial y_{14}},\\
X&=-u\sin\! x\frac{\partial}{\partial y_{11}}-v\sin\! x\frac{\partial}{\partial y_{12}}+u\cos\! x\frac{\partial}{\partial y_{13}}+v\cos\! x\frac{\partial}{\partial y_{14}},\\
Z&=-u\sin\! z\frac{\partial}{\partial y_{1}}-v\sin\! z\frac{\partial}{\partial y_{2}}+u\cos\! z\frac{\partial}{\partial y_{5}}+v\cos\! z\frac{\partial}{\partial y_{6}}+\frac{\partial}{\partial y_{9}},\\
W&=-u\sin\! w\frac{\partial}{\partial y_{3}}-v\sin\! w\frac{\partial}{\partial y_{4}}+u\cos\! w\frac{\partial}{\partial y_{7}}+v\cos\! w\frac{\partial}{\partial y_{8}}+\frac{\partial}{\partial y_{10}},
\end{align*}
where $(y_{1},...,y_{14})$ are natural coordinates of $\mathbb{R}^{14}$.
Then $\mathcal{D}^{T}=span\{U,V\}$ is a holomorphic, $\mathcal{D}^{\bot}=span\{X\}$ is a totally real and
$\mathcal{D}^{\theta}=span\{Z,W\}$ is a (proper) pointwise slant distribution with the slant function $\theta=\cos^{-1}(\frac{1}{1+u^{2}+v^{2}}).$
Thus, $M$ is a generalized $J$-induced submanifold of order 1 of $\mathbb{R}^{14}.$ Also, one can easily see the distribution
$\mathcal{D}^{T}$ is totally geodesic and the distributions $\mathcal{D}^{\bot}$ and $\mathcal{D}^{\theta}$ are integrable.  If we denote the integral submanifolds of
$\mathcal{D}^{T}, \mathcal{D}^{\bot}$ and $\mathcal{D}^{\theta}$ by $M_{T}, M_{\bot}$  and $M_{\theta}$, respectively,
then the induced metric tensor of $M$ is
\begin{equation}
\label{}\nonumber
\begin{array}{lll}
ds^{2}&=&3(du^{2}+dv^{2})+(u^{2}+v^{2})dx^{2}+(1+u^{2}+v^{2})(dz^{2}+dw^{2})\\
&=&g_{M_{T}}+(u^{2}+v^{2})g_{M_{\bot}}+(1+u^{2}+v^{2})g_{M_{\theta}}.
\end{array}
\end{equation}
Thus, $M=M_{T}\times_{f}M_{\bot}\times_{\sigma}M_{\theta}$ is a (non-trivial) biwarped product
generalized $J$-induced submanifold of order 1 of $\mathbb{R}^{14}$ with warping functions $f=\sqrt{u^{2}+v^{2}}$ and $\sigma=\sqrt{1+u^{2}+v^{2}}.$
\end{example}
\section{Biwarped product generalized $J$-induced submanifolds of order 1\\
in the form $M_{T}\times_{f}M_{\bot}\times_{\sigma}M_{\theta}$ of a K\"{a}hler  manifold}
In this section, we give a characterization for biwarped product generalized $J$-induced submanifolds of order 1
in the form  $M_{T}\times_{f}M_{\bot}\times_{\sigma}M_{\theta}$, where $M_{T}$ is a holomorphic, $M_{\bot}$ is a totally real and $M_{\theta}$ is
a pointwise slant submanifold of a K\"{a}hler  manifold $(\bar{M},J,g)$. After that,
we investigate the behavior of the second fundamental form of such submanifolds and as a result,
we give a necessary and sufficient condition for such manifolds to be locally trivial.
Now, we give one of the main theorems of this paper. We first recall the following fact given in \cite{Di} to prove our theorem.
\begin{remark} \label{rema1} (Remark 2.1 \cite{Di}) Suppose that the tangent bundle of a Riemannian manifold $M$ splints into an orthogonal sum
$TM=\mathcal{D}_{0}\oplus\mathcal{D}_{1}\oplus...\oplus\mathcal{D}_{k}$ of non-trivial distributions such that each $\mathcal{D}_{j}$ is spherical
and its complement in $TM$ is autoparallel for $j\in\{1,2,...,k\}.$ Then the manifold $M$ is locally isometric to a
multiply warped product $M_{0}\times_{f_{1}}M_{1}\times...\times_{f_{k}}M_{k}.$
\end{remark}
\begin{theorem} Let $M$ be a proper generalized $J$-induced submanifold of order 1 of a K\"{a}hler  manifold $(\bar{M},J,g).$
Then $M$ is a locally biwarped product submanifold of the form $M_{T}\times_{f}M_{\bot}\times_{\sigma}M_{\theta}$ if and only if
\begin{align}
&A_{JX}V=-JV(\lambda)X, \label{g16}\\
&A_{FTZ}V-A_{FZ}JV=-\sin^{2}\!\theta V(\mu)Z, \label{g17}
\end{align}
for some functions $\lambda$ and $\mu$ satisfying $X(\lambda)=Z(\lambda)=0$ and $X(\mu)=Z(\mu)=0,$ and
\begin{align}
&g(A_{JY}TZ,X)=g(A_{FTZ}Y,X),\label{k1}\\
&g(A_{JY}TW,Z)=g(A_{FTZ}Y,W). \label{k2}
\end{align}
for $V\in\mathcal{D}^{T}$,\, $X,Y\in\mathcal{D}^\bot$ and  $Z,W\in\mathcal{D}^\theta$.
\end{theorem}
\begin{proof} Let $M$ be a biwarped product proper generalized $J$-induced submanifold of order 1 of a K\"{a}hler  manifold $(\bar{M},J,g)$
in the form $M_{T}\times_{f}M_{\bot}\times_{\sigma}M_{\theta}$. Then for any $V\in\mathcal{D}^{T}$, $X\in\mathcal{D}^\bot$ and  $Z\in\mathcal{D}^\theta$
using (\ref{e2})$\sim$(\ref{e3}) and  (\ref{e5})$\sim$(\ref{e6}),  we have
\begin{align*}
g(A_{JX}V,U)=-g(\bar{\nabla}_{U}JX,V)=g(\bar{\nabla}_{U}X,JV)=g(\nabla_{U}X,JV).
\end{align*}
Here, we know $\nabla_{U}X=U(\ln f)X$ from (\ref{a2}). Thus, we obtain
\begin{equation} \label{g18}
\begin{array}{c}
g(A_{JX}V,U)=U(\ln f)g(X,JV)=0
\end{array},
\end{equation}
since $g(X,JV)=0.$ Similarly, we have
\begin{align*}
g(A_{JX}V,Z)=-g(\bar{\nabla}_{Z}JX,V)=g(\bar{\nabla}_{Z}X,JV)=g(\nabla_{Z}X,JV).
\end{align*}
Thus, we obtain
\begin{equation} \label{g19}
\begin{array}{c}
g(A_{JX}V,Z)=0
\end{array},
\end{equation}
since $\nabla_{Z}X=0$ from (\ref{a3}). Next, by a similar argument, for $Y\in\mathcal{D}^\bot$, we have
\begin{align*}
g(A_{JX}V,Y)=-g(\bar{\nabla}_{Y}JX,V)=g(\bar{\nabla}_{Y}X,JV)=-g(\nabla_{Y}JV,X).
\end{align*}
Using (\ref{a2}), we obtain
\begin{equation} \label{g20}
\begin{array}{c}
g(A_{JX}V,Y)=g(-JV(\ln f)X,Y)
\end{array}.
\end{equation}
Moreover, we have $X(\ln f)=Z(\ln f)=0$, since $f$ depends on only points of $M_{T}.$ So, we conclude that $\lambda=\ln f.$
Thus, ffrom (\ref{g18})$\sim$(\ref{g20}), it follows that (\ref{g16}). Now,
we prove (\ref{g17}). Using (\ref{e2})$\sim$(\ref{e3}) and (\ref{e5})$\sim$(\ref{e8}), we have
\begin{align*}
g(A_{FTZ}V-A_{FZ}JV,U)=&-g(\bar{\nabla}_{V}FTZ,U)+g(\bar{\nabla}_{JV}FZ,U)\\
=&-g(\bar{\nabla}_{V}JTZ,U)+g(\bar{\nabla}_{V}T^{2}Z,U)\\
&+g(\bar{\nabla}_{JV}JZ,U)-g(\bar{\nabla}_{JV}TZ,U)\\
=&+g(\bar{\nabla}_{V}TZ,JU)-\cos^{2}\!\theta g(\bar{\nabla}_{V}Z,U)\\
&-g(\bar{\nabla}_{JV}Z,JU)-g(\bar{\nabla}_{JV}TZ,U)\\
=&+g(\nabla_{V}TZ,JU)-\cos^{2}\!\theta g(\nabla_{V}Z,U)\\
&-g(\nabla_{JV}Z,JU)-g(\nabla_{JV}TZ,U).\qquad\qquad
\end{align*}
Here, if we use (\ref{a2}), we obtain
\begin{align*}
g(A_{FTZ}V-A_{FZ}JV,U)=&V(\ln \sigma)g(TZ,JU)-\cos^{2}\!\theta JV(\ln \sigma)g(Z,U)\\
&-JV(\ln \sigma)g(Z,JU)-V(\ln \sigma)g(TZ,U).
\end{align*}
Hence, we get
\begin{equation} \label{g21}
\begin{array}{c}
g(A_{FTZ}V-A_{FZ}JV,U)=0
\end{array},
\end{equation}
since $g(TZ,JU)=g(Z,U)=g(Z,JU)=g(TZ,U)=0.$ Similarly, we have
\begin{align*}
g(A_{FTZ}V-A_{FZ}JV,X)=&-g(\bar{\nabla}_{X}FTZ,V)+g(\bar{\nabla}_{X}FZ,JV)\\
=&-g(\bar{\nabla}_{X}JTZ,V)+g(\bar{\nabla}_{X}T^{2}Z,V)\\
&+g(\bar{\nabla}_{X}JZ,JV)-g(\bar{\nabla}_{X}TZ,JV)\\
=&+g(\bar{\nabla}_{X}TZ,JV)-\cos^{2}\!\theta g(\bar{\nabla}_{X}Z,V)\\
&-g(\bar{\nabla}_{X}Z,V)-g(\bar{\nabla}_{X}TZ,JV)\\
=&+g(\nabla_{X}TZ,JV)-\cos^{2}\!\theta g(\nabla_{X}Z,V)\\
&-g(\nabla_{X}Z,V)-g(\nabla_{X}TZ,JV).\qquad\qquad
\end{align*}
Here, we know $\nabla_{X}TZ=\nabla_{X}Z=0$ from (\ref{a3}). So, we get
\begin{equation} \label{g22}
\begin{array}{c}
g(A_{FTZ}V-A_{FZ}JV,X)=0
\end{array}.
\end{equation}
On the other hand, using (\ref{g4}) , we have
\begin{equation} \nonumber
\begin{array}{c}
g(A_{FTZ}V-A_{FZ}JV,W)=-\sin^{2}\!\theta g(\nabla_{W}V,Z)
\end{array}.
\end{equation}
Using (\ref{a2}), we get
\begin{equation}  \label{g23}
\begin{array}{c}
g(A_{FTZ}V-A_{FZ}JV,W)=g(-\sin^{2}\!\theta V(\ln \sigma)Z,W)
\end{array}.
\end{equation}
Moreover, we have $X(\ln \sigma)=Z(\ln \sigma)=0$, since $\sigma$ depends on only points of $M_{T}.$ So, we conclude that $\mu=\ln \sigma.$
Thus, (\ref{g17}) follows form (\ref{g21}), (\ref{g22}) and (\ref{g23}). Next, we prove (\ref{k1}) and (\ref{k2}).
Using (\ref{g10}) and (\ref{a3}), we have
\begin{equation} \nonumber
\begin{array}{c}
g(A_{JY}X,TZ)-g(A_{FTZ}X,Y)=-\cos^{2}\!\theta g(\nabla_{X}Z,Y)=0
\end{array}.
\end{equation}
Thus, (\ref{k1}) follows. Similarly, using (\ref{g18}) and (\ref{a3}), we have
\begin{equation} \nonumber
\begin{array}{c}
g(A_{JY}TW,Z)=g(A_{FTZ}Y,W)=-\cos^{2}\!\theta g(\nabla_{Z}Y,W)=0
\end{array}.
\end{equation}
So, we get (\ref{k2}).\\

Conversely, assume that $M$ is a proper generalized $J$-induced submanifold of order 1 of a K\"{a}hler  manifold $(\bar{M},J,g)$
such that (\ref{g16})$\sim$(\ref{k2}) hold. Then we satisfy (\ref{g12}) and (\ref{g13}) by using (\ref{g16}) and (\ref{g17}), respectively.
Thus, by Theorem \ref{thm01}, the holomorphic distribution $\mathcal{D}^{T}$ is totally geodesic and as a result it is integrable.
By (\ref{k1}) and (\ref{k2}), we easily satisfy (\ref{g14}) and (\ref{g15}). Thus, by Theorem \ref{thm02},
the pointwise slant distribution $\mathcal{D}^\theta$ is integrable. Also, by Remark \ref{rema0},
the totally real  distribution $\mathcal{D}^\bot$ is always integrable. Let $M_{T}$, $M_{\bot}$ and $M_{\theta}$ be the
integral manifolds of $\mathcal{D}^{T}$, $\mathcal{D}^\bot$ and $\mathcal{D}^\theta$, respectively. If we denote the second fundamental form of
$M_{\bot}$ in $M$ by $h^{\bot}$, for $X,Y\in\mathcal{D}^\bot$ and  $Z\in\mathcal{D}^\theta$, using (\ref{e2}), (\ref{g10}) and (\ref{k1}),
we have
\begin{equation}  \label{g24}
\begin{array}{c}
g(h^{\bot}(X,Y),Z)=g(\nabla_{X}Y,Z)=0
\end{array}.
\end{equation}
For any $X,Y\in\mathcal{D}^\bot$ and $V\in\mathcal{D}^T$, using (\ref{e2}), (\ref{g6}) and (\ref{g16}), we have
\begin{equation}  \nonumber
\begin{array}{c}
g(h^{\bot}(X,Y),V)=g(\nabla_{X}Y,V)=-g(A_{JY}X,JV)=-V(\lambda)g(X,Y)
\end{array}.
\end{equation}
After some calculation, we obtain
\begin{equation}  \label{g25}
\begin{array}{c}
g(h^{\bot}(X,Y),V)=g(-g(X,Y)\nabla\lambda,V)
\end{array},
\end{equation}
where $\nabla\lambda$ is the gradient of $\lambda$. Thus, from (\ref{g24}) and (\ref{g25}), we conclude that
\begin{equation}  \nonumber
\begin{array}{c}
h^{\bot}(X,Y)=-g(X,Y)\nabla\lambda
\end{array}.
\end{equation}
This equation says that $M_{\bot}$ is totally umbilic in $M$ with the mean curvature vector field $-\nabla\lambda.$
Now, we show that $-\nabla\lambda$ is parallel. We have to satisfy $g(\nabla_{X}\nabla\lambda,E)=0$
for $X\in\mathcal{D}^\bot$ and $E\in(\mathcal{D}^\bot)^{\bot}=\mathcal{D}^T\oplus\mathcal{D}^\theta.$
Here, we can put $E=V+Z,$ where $V\in\mathcal{D}^T$ and $Z\in\mathcal{D}^\theta.$ By direct computations, we obtain
\begin{align*}
g(\nabla_{X}\nabla\lambda,E)=&\{Xg(\nabla\lambda,E)-g(\nabla\lambda,\nabla_{X}E)\}\\
=&\{X(E(\lambda))-[X,E]\lambda-g(\nabla\lambda,\nabla_{E}X)\}\\
=&\{[X,E]\lambda+E(X(\lambda))-[X,E]\lambda-g(\nabla\lambda,\nabla_{E}X)\}\\
=&-g(\nabla\lambda,\nabla_{V}X)-g(\nabla\lambda,\nabla_{Z}X),
\end{align*}
since $X(\lambda)=0.$ Here, for any $U\in\mathcal{D}^T$, we have $g(\nabla_{V}X,U)=-g(\nabla_{V}U,X)=0,$
since $M_{T}$ is totally geodesic in $M$. Thus, $\nabla_{V}X\in\mathcal{D}^\bot$ or $\nabla_{V}X\in\mathcal{D}^\theta.$
In either case, we have
\begin{equation} \label{g26}
\begin{array}{c}
g(\nabla\lambda,\nabla_{V}X)=0
\end{array}.
\end{equation}
On the other hand, from (\ref{g9}), we have $g(\nabla_{Z}X,U)=-g(A_{JX}JU,Z).$ Here, using (\ref{k1}),
we obtain  $g(\nabla_{Z}X,U)=0.$ That is; $\nabla_{Z}X\in\mathcal{D}^\bot$ or $\nabla_{Z}X\in\mathcal{D}^\theta.$
In either case, we get
\begin{equation} \label{g27}
\begin{array}{c}
g(\nabla\lambda,\nabla_{Z}X)=0
\end{array}.
\end{equation}
From (\ref{g26}) and  (\ref{g27}), we find
\begin{equation} \nonumber
\begin{array}{c}
g(\nabla_{X}\nabla\lambda, E)=0
\end{array}.
\end{equation}
Thus, $M_{\bot}$ is spherical, since it is also totally umbilic. Consequently, $\mathcal{D}^\bot$ is spherical.\\
Next, we show that $\mathcal{D}^\theta$ is spherical. Let $h^{\theta}$ denote the second fundamental form of $M_{\theta}$
in $M.$ Then for $Z,W\in\mathcal{D}^\theta$ and  $X\in\mathcal{D}^\bot$, using (\ref{e2}), (\ref{g8}) and (\ref{k1}),
we have
\begin{equation}  \label{g28}
\begin{array}{c}
g(h^{\theta}(Z,W),X)=g(\nabla_{Z}W,X)=0
\end{array}.
\end{equation}
On the other hand, for any $V\in\mathcal{D}^T$, using (\ref{e2}) and (\ref{g4}), we have
\begin{equation}  \nonumber
\begin{array}{c}
g(h^{\theta}(Z,W),V)=g(\nabla_{Z}W,V)=\csc^{2}\!\theta g(A_{FTW}V-A_{FW}JV,Z)
\end{array}.
\end{equation}
Using (\ref{g17}), we obtain
\begin{equation} \nonumber
\begin{array}{c}
g(h^{\theta}(Z,W),V)=-V(\mu)g(Z,W)
\end{array}.
\end{equation}
After some calculation, we get
\begin{equation} \label{g29}
\begin{array}{c}
g(h^{\theta}(Z,W),V)=-g(g(Z,W)\nabla\mu,V)
\end{array},
\end{equation}
where $\nabla\mu$ is the gradient of $\mu$. Thus, from (\ref{g28}) and (\ref{g29}), we deduce that
\begin{equation}  \nonumber
\begin{array}{c}
h^{\theta}(Z,W)=-g(Z,W)\nabla\mu
\end{array}.
\end{equation}
It means that $M_{\theta}$ is totally umbilic in $M$ with the mean curvature vector field $-\nabla\mu.$
What's left is to show that $-\nabla\mu$ is parallel. We have to satisfy $g(\nabla_{Z}\nabla\mu,E)=0$
for $Z\in\mathcal{D}^\theta$ and $E\in(\mathcal{D}^\theta)^{\bot}=\mathcal{D}^T\oplus\mathcal{D}^\bot.$
Here, $E=V+X,$ for $V\in\mathcal{D}^T$ and $X\in\mathcal{D}^\bot.$ Upon direct calculation, we obtain
\begin{align*}
g(\nabla_{Z}\nabla\mu,E)=&\{Zg(\nabla\mu,E)-g(\nabla\mu,\nabla_{Z}E)\}\\
=&\{Z(E(\mu))-[Z,E]\mu-g(\nabla\mu,\nabla_{E}Z)\}\\
=&\{[Z,E]\mu+E(Z(\mu))-[Z,E]\mu-g(\nabla\mu,\nabla_{E}Z)\}                                     \\
=&-g(\nabla\mu,\nabla_{V}Z)-g(\nabla\mu,\nabla_{X}Z),
\end{align*}
since $Z(\mu)=0.$ Here, for any $U\in\mathcal{D}^T$, using (\ref{g16}) and (\ref{g11}), we have
\begin{equation}  \nonumber
\begin{array}{c}
g(\nabla_{X}Z,U)=-\csc^{2}\!\theta \{g(A_{FZ}JU-A_{FTZ}U,X)\}
\end{array}.
\end{equation}
So, $\nabla_{X}Z\in\mathcal{D}^\bot$ or $\nabla_{V}X\in\mathcal{D}^\theta.$ Hence
\begin{equation} \label{g30}
\begin{array}{c}
g(\nabla\mu,\nabla_{X}Z)=0
\end{array},
\end{equation}
since $\nabla\mu\in\mathcal{D}^T.$ On the other hand, since $M_{T}$ is totally geodesic in $M$, we have $g(\nabla_{V}Z,U)=-g(\nabla_{V}U,Z)=0.$
Hence, $\nabla_{V}Z\in\mathcal{D}^\bot$ or $\nabla_{V}Z\in\mathcal{D}^\theta.$ So, we get
\begin{equation} \label{g31}
\begin{array}{c}
g(\nabla\mu,\nabla_{V}Z)=0
\end{array}.
\end{equation}
By (\ref{g30}) and  (\ref{g31}), we find
\begin{equation} \nonumber
\begin{array}{c}
g(\nabla_{Z}\nabla\mu, E)=0
\end{array}.
\end{equation}
Lastly, we prove that $(\mathcal{D}^\bot)^{\bot}=\mathcal{D}^T\oplus\mathcal{D}^\theta$ and $(\mathcal{D}^\theta)^{\bot}=\mathcal{D}^T\oplus\mathcal{D}^\bot$
are autoparallel. In fact, $\mathcal{D}^T\oplus\mathcal{D}^\theta$ is autoparallel if and only if all four types of covariant derivatives
$\nabla_{U}V, \nabla_{U}Z, \nabla_{Z}U, \nabla_{Z}W$ are again in $\mathcal{D}^T\oplus\mathcal{D}^\bot$ for $U,V\in\mathcal{D}^T$ and $X\in\mathcal{D}^\bot.$ This is equivalent to say that all four inner products $g(\nabla_{U}V,X), g(\nabla_{U}Z,X), g(\nabla_{Z}U,X)$,
$g(\nabla_{Z}W, X)$ vanish, where $U,V\in\mathcal{D}^T$,\, $Z,W\in\mathcal{D}^\theta$ and $X\in\mathcal{D}^\bot.$
Using (\ref{g5}), (\ref{g5}) and (\ref{g16}), we get
\begin{equation} \nonumber
\begin{array}{c}
g(\nabla_{U}V,X)=g(\nabla_{Z}U,X)=0
\end{array}.
\end{equation}
By (\ref{g6}), (\ref{g8}) and (\ref{k1}), we get
\begin{equation} \nonumber
\begin{array}{c}
g(\nabla_{U}Z,X)=g(\nabla_{Z}W,X)=0
\end{array}.
\end{equation}
Thus, $\mathcal{D}^T\oplus\mathcal{D}^\theta$ is autoparallel. On the other hand,
$\mathcal{D}^T\oplus\mathcal{D}^\bot$ is autoparallel if and only if
all four inner products $g(\nabla_{U}V,Z), g(\nabla_{U}X,Z), g(\nabla_{X}U,Z), g(\nabla_{X}Y,Z)$
vanish. Firstly, we have already $g(\nabla_{U}X,Z)=0$ from above. Using (\ref{g3}), (\ref{g11}) and (\ref{g17}), we get
\begin{equation} \nonumber
\begin{array}{c}
g(\nabla_{U}V,Z)=g(\nabla_{X}U,Z)=0
\end{array}.
\end{equation}
Using (\ref{g10}) and (\ref{k1}), we find
\begin{equation} \nonumber
\begin{array}{c}
g(\nabla_{X}Y,Z)=0
\end{array}.
\end{equation}
So, $\mathcal{D}^T\oplus\mathcal{D}^\bot$ is autoparallel.
Thus, by Remark \ref{rema1}, $M$ is locally biwarped product submanifold of the form $M_{T}\times_{f}M_{\bot}\times_{\sigma}M_{\theta}.$
\end{proof}
Let $M=M_{T}\times_{f}M_{\bot}\times_{\sigma}M_{\theta}$ be a non-trivial biwarped product generalized
$J$-induced submanifold of order 1 of a K\"{a}hler manifold $(\bar{M},J,g).$ Next, we  investigate the behavior of the second fundamental form $h$ of such submanifolds.
\begin{lemma} \label{lem1} Let $M$ be a biwarped product generalized $J$-induced submanifold of order 1 in the form $M_{T}\times_{f}M_{\bot}\times_{\sigma}M_{\theta}$ of a K\"{a}hler  manifold $(\bar{M},J,g)$. Then for the second fundamental form $h$ of $M$ in $\bar{M}$, we have
\begin{align}
&g(h(U,V),JX)=0\,,\qquad\qquad\label{e12} \\
&g(h(V,X),JY)=-JV(\ln f)g(X,Y)\,,\label{e13} \\
&g(h(V,Z),JX)=0\,,\qquad\qquad\qquad\qquad\label{e14}
\end{align}
where $U,V\in\mathcal{D}^T$,\, $X,Y\in\mathcal{D}^\bot$ and  $Z\in\mathcal{D}^\theta$.
\end{lemma}
\begin{proof} Using (\ref{e2}),(\ref{e5}) and (\ref{e6}), we have
\begin{equation} \nonumber
\begin{array}{c}
g(h(U,V),JX)=g(\bar{\nabla}_{U}V,JX)=-g(\bar{\nabla}_{U}JX,V)=g(\bar{\nabla}_{U}X,JV)
\end{array}
\end{equation}
for $U,V\in\mathcal{D}^T$ and $X\in\mathcal{D}^\bot$. Again, using (\ref{e2}),
we obtain $g(h(U,V),JX)=g(\nabla_{U}X,JV)$. Here, we know $\nabla_{U}X=U(\ln f)X$ from (\ref{a2}). Thus, we obtain
$g(h(U,V),JX)=U(\ln f)g(X,V)=0$, since $g(X,V)=0.$ So, (\ref{e12}) follows. Now, let
$V\in\mathcal{D}^T$ and $X,Y\in\mathcal{D}^\bot$. Then using (\ref{e2}), (\ref{e5}) and (\ref{e6}), we have
\begin{equation} \nonumber
\begin{array}{c}
g(h(V,X),JY)=g(\bar{\nabla}_{X}V,JY)=-g(\bar{\nabla}_{X}JV,Y)=-g(\nabla_{X}JV,Y)
\end{array}.
\end{equation}
Here, we know  $\nabla_{X}JV=JV(\ln f)X$ from (\ref{a2}). Thus, we get (\ref{e13}). The last assertion can be obtained in a similar method.
\end{proof}
The previous lemma shows partially us the behavior of the second fundamental form $h$ of the biwarped product generalized
$J$-induced submanifolds of order 1 has the form $M_{T}\times_{f}M_{\bot}\times_{\sigma}M_{\theta}$ in the normal subbundle $J(\mathcal{D}^T).$
\begin{lemma} \label{lem2} Let $M$ be a biwarped product generalized $J$-induced submanifold of order 1 in the form $M_{T}\times_{f}M_{\bot}\times_{\sigma}M_{\theta}$ of a K\"{a}hler manifold $(\bar{M},J,g)$. Then for the second fundamental form $h$ of $M$ in $\bar{M}$, we have
\begin{align}
&g(h(U,V),FZ)=0\,,\label{e15} \\
&g(h(V,X),FZ)=0\,,\label{e16} \\
&g(h(V,Z),FW)=-JV(\ln \sigma)g(Z,W)-V(\ln \sigma)g(Z,TW)\,,\qquad\qquad\label{e17}
\end{align}
where $U,V\in\mathcal{D}^T$,\, $X\in\mathcal{D}^\bot$ and  $Z,W\in\mathcal{D}^\theta$.
\end{lemma}
\begin{proof} Using (\ref{e2}),(\ref{e5}) and (\ref{e6}), we have
\begin{equation}\nonumber
\begin{array}{c}
g(h(U,V),FZ)=g(\bar{\nabla}_{U}V,FZ)=g(\bar{\nabla}_{U}V,JZ)-g(\bar{\nabla}_{U}V,TZ)
\end{array}
\end{equation}
for $U,V\in\mathcal{D}^T$ and $Z\in\mathcal{D}^\theta$. After some calculation, we obtain
\begin{equation}\nonumber
\begin{array}{c}
g(h(U,V),FZ)=g(\nabla_{U}Z,JV)+ g(\nabla_{U}TZ,V)
\end{array}.
\end{equation}
Here, we know $\nabla_{U}Z=U(\ln \sigma)Z$ and $\nabla_{U}Z=U(\ln \sigma)TZ$ from (\ref{a2}). Thus, we get
\begin{equation}\nonumber
\begin{array}{c}
g(h(U,V),FZ)=U(\ln \sigma)g(Z,JV)+U(\ln \sigma)g(TZ,V)=0
\end{array},
\end{equation}
since $g(Z,JV)=g(TZ,V)=0.$ So, (\ref{e15}) follows.
Now, we prove (\ref{e17}). Let $V\in\mathcal{D}^T$ and $Z,W\in\mathcal{D}^\theta$. Then using (\ref{e2}),(\ref{e5}) and (\ref{e6}), we have
\begin{equation}\nonumber
\begin{array}{c}
g(h(V,Z),FW)=-g(\nabla_{Z}JV,W)-g(\nabla_{Z}V,TW)
\end{array}.
\end{equation}
Again, by (\ref{a2}), we easily get (\ref{e17}).
Similarly, we can obtain (\ref{e16}).
\end{proof}
The last lemma shows partially us the behavior of the second fundamental form $h$ of the biwarped product generalized
$J$-induced submanifolds of order 1 of type $M_{T}\times_{f}M_{\bot}\times_{\sigma}M_{\theta}$ in the normal subbundle $F(\mathcal{D}^\theta).$
\begin{remark} Equation (\ref{e12}) of Lemma \ref{lem1} and equation (\ref{e15}) of Lemma \ref{lem2} also hold
for skew CR-warped product submanifolds of K\"{a}hler  manifolds, see Lemma 4 of \cite{S4}.
Moreover, equations (\ref{e15}) and (\ref{e17}) of Lemma \ref{lem2} are also valid for warped product pointwise semi-slant submanifolds
in K\"{a}hler  manifolds, see Lemma 5.3 of \cite{S5}.
\end{remark}
By (\ref{e12}) and (\ref{e15}), we immediately get the following result.
\begin{cor} Let $M$ be a biwarped product generalized $J$-induced submanifold of order 1
in the form $M_{T}\times_{f}M_{\bot}\times_{\sigma}M_{\theta}$ of a K\"{a}hler  manifold $(\bar{M},J,g)$
such that the invariant normal subbundle $\overline{\mathcal{D}}^T=\{0\}.$ Then  $M$ is $\mathcal{D}^T$-geodesic.
\end{cor}
Lastly, we give a necessary and sufficient condition for such submanifolds to be locally trivial.
\begin{theorem} Let $M$ be a biwarped product proper generalized $J$-induced submanifold in the form $M_{T}\times_{f}M_{\bot}\times_{\sigma}M_{\theta}$
of a K\"{a}hler  manifold $(\bar{M},J,g)$ such that the invariant normal subbundle $\overline{\mathcal{D}}^T=\{0\}.$
Then  $M$ is locally trivial if and only if $M$ is  both $(\mathcal{D}^T,\mathcal{D}^\bot)$ and  $(\mathcal{D}^T,\mathcal{D}^\theta)$-mixed geodesic.
\end{theorem}
\begin{proof} Let $M$ be a biwarped product proper generalized $J$-induced submanifold of order 1 in the form $M_{T}\times_{f}M_{\bot}\times_{\sigma}M_{\theta}$
of a K\"{a}hler  manifold $(\bar{M},J,g)$ such that the invariant normal subbundle $\overline{\mathcal{D}}^T=\{0\}.$
If $M$ is locally trivial, then the warping functions $f$ and $\sigma$ are constants.
By (\ref{e13}), we have $g(h(V,X),JY)=0$ for $V\in\mathcal{D}^T$ and $X,Y\in\mathcal{D}^\bot$, since $JV(\ln f)=0.$
Taking into account the equation (\ref{g2}) and (\ref{e16}), we get $h(V,X)=0$.
It means that $M$ is $(\mathcal{D}^T,\mathcal{D}^\bot)$-mixed geodesic.\\
On the other hand, for any $V\in\mathcal{D}^T$ and  $Z,W\in\mathcal{D}^\theta,$ we have $g(h(V,Z),FW)=0$ from (\ref{e17}),
since $JV(\ln \sigma)=0$ and $V(\ln \sigma)=0.$ Taking into account the equation (\ref{g2}) and (\ref{e14}), we obtain
$h(V,Z)=0$. Which says us $M$ is $(\mathcal{D}^T,\mathcal{D}^\theta)$-mixed geodesic.\\

Conversely, let $M$ be both $(\mathcal{D}^T,\mathcal{D}^\bot)$ and $(\mathcal{D}^T,\mathcal{D}^\theta)$-mixed geodesic.
Then, for any $V\in\mathcal{D}^T$, from (\ref{e13}) we conclude that $JV(\ln f)=0$, since $M$ is
$(\mathcal{D}^T,\mathcal{D}^\bot)$-mixed geodesic. Hence, it follows that $f$ is a constant. Since
$M$ is also $(\mathcal{D}^T,\mathcal{D}^\theta)$-mixed geodesic, for $V\in\mathcal{D}^T$ and  $Z,W\in\mathcal{D}^\theta,$ we have
\begin{equation}
\label{e18}
\begin{array}{c}
JV(\ln \sigma)g(Z,W)+V(\ln \sigma)g(Z,TW)=0
\end{array}
\end{equation}
from (\ref{e17}). If we put $V=JV$ in (\ref{e18}), we obtain
\begin{equation}
\label{} \nonumber
\begin{array}{c}
-V(\ln \sigma)g(Z,W)+JV(\ln \sigma)g(Z,TW)=0
\end{array}.
\end{equation}
If we take $W=TW$ in the last equation and use (\ref{e8}), the last equation becomes
\begin{equation}
\label{e19}
\begin{array}{c}
-V(\ln \sigma)g(Z,TW)-\cos^{2}\!\theta JV(\ln \sigma)g(Z,W)=0
\end{array}.
\end{equation}
From (\ref{e18}) and (\ref{e19}), we get
\begin{equation}
\label{e20}
\begin{array}{c}
\sin^{2}\!\theta JV(\ln \sigma)g(Z,W)=0
\end{array}.
\end{equation}
Since $M$ is proper, $\sin\!\theta\neq0.$ So, we deduce that $JV(\ln \sigma)=0$  from (\ref{e20}).
Hence, it follows that $\sigma$ is a constant. Thus, $M$ must be locally trivial, since we found the warping functions
$f$ and $\sigma$ as constants.
\end{proof}
\section{An inequality for non-trivial biwarped product generalized\\
$J$-induced submanifolds of order 1 in the form $M_{T}\times_{f}M_{\bot}\times_{\sigma}M_{\theta}$}
In this section, by using the results given the preceding section, we shall establish an inequality for the squared
norm of the second fundamental form in terms of the warping functions for biwarped product generalized $J$-induced submanifolds of order 1
in the form $M_{T}\times_{f}M_{\bot}\times_{\sigma}M_{\theta}$,  
where $M_{T}$ is a holomorphic, $M_{\bot}$ is a totally real and $M_{\theta}$ is
a pointwise slant submanifold of a K\"{a}hler  manifold$(\bar{M},J,g)$.\\

Let $M$ be a $(k+n+m)$-dimensional biwarped product generalized $J$-induced submanifold of order 1 of type
$M_{T}\times_{f}M_{\bot}\times_{\sigma}M_{\theta}$ of a K\"{a}hler  manifold $\bar{M}$ and let
$\{e_{1},...,e_{k},\tilde{e}_{1},...,\tilde{e}_{n}, \bar{e}_{1},...,\bar{e}_{m}, e_{1}^{*},...,e_{k}^{*}, J\tilde{e}_{1},...,J\tilde{e}_{n},\hat{e_{1}},...,\hat{e_{l}}\}$
be a canonical orthonormal basis of $\bar{M}$ such that $\{e_{1},...,e_{k}\}$
is an orthonormal basis of $\mathcal{D}^T$, $\{\tilde{e}_{1},...,\tilde{e}_{n}\}$ is an orthonormal basis of
$\mathcal{D}^\bot$, $\{\bar{e}_{1},...,\bar{e}_{k}\}$ is an orthonormal basis of $\mathcal{D}^\theta$, $\{J\tilde{e}_{1},...,J\tilde{e}_{n}\}$ is
an orthonormal basis of $J\mathcal{D}^\bot$, $\{e_{1}^{*},...,e_{m}^{*}\}$ is an orthonormal basis of
$F\mathcal{D}^\theta$ and $\{\hat{e_{1}},...,\hat{e_{l}}\}$ is an orthonormal basis of $\overline{\mathcal{D}}^T.$
Here, $k=dim(\mathcal{D}^T)$, $n=dim(\mathcal{D}^\bot)$, $m=dim(\mathcal{D}^\theta)$ and $l=dim(\overline{\mathcal{D}}^T).$
\begin{remark} \label{rem2} In view of (\ref{e5}), we can observe that $\{Je_{1},...,Je_{k}\}$ is
also an orthonormal basis of $\mathcal{D}^T$. On the other hand,
with the help of (\ref{e9}) and (\ref{e10}), we can
see that $\{\sec\!\theta T\bar{e}_{1},...,\sec\!\theta
T\bar{e}_{m}\}$ is also an orthonormal basis of $\mathcal{D}^\theta$
and $\{\csc\!\theta F\bar{e}_{1},...,\csc\!\theta F\bar{e}_{m}\}$ is
also an orthonormal basis of $F\mathcal{D}^\theta$, where $\theta$ is the slant function of $\mathcal{D}^\theta$.
\end{remark}
\begin{theorem} \label{thm1}
Let $M$ be a biwarped product proper generalized $J$-induced submanifold of order 1 in the form $M_{T}\times_{f}M_{\bot}\times_{\sigma}M_{\theta}$
of a K\"{a}hler manifold $(\bar{M},J,g)$. Then the squared norm of the second fundamental form $h$ of $M$ satisfies
\begin{align} \label{e21}
\|h\|^{2}\geq2\{n\|\nabla(\ln f)\|^{2}+m(\csc^{2}\!\theta+\cot^{2}\!\theta)\|\nabla(\ln \sigma)\|^{2}\}
\end{align}
where $n=dim(M_{\bot})$ and $m=dim(M_{\theta}).$ The equality case of (\ref{e21}) holds identically if and only if the following assertions are true.\\
\textbf{a)} $M_{T}$ is a totally geodesic submanifold in $\bar{M}.$\\
\textbf{b)} $M_{\bot}$ and $M_{\theta}$ are totally umbilic submanifolds in $\bar{M}$
with their mean curvature vector fields $-\nabla(\ln f)$ and $-\nabla(\ln \sigma)$, respectively.\\
\textbf{c)} $M$ is minimal in $\bar{M}.$\\
\textbf{d)} $M$ is $(\mathcal{D}^\bot,\mathcal{D}^\theta)$-mixed geodesic.
\end{theorem}
\begin{proof} By the decomposition (\ref{g1}), the squared norm of the
second fundamental form $h$ can be written as
\begin{align*}
 \|h\|^{2}=&\,\|h(\mathcal{D}^T,\mathcal{D}^T)\|^{2}+\|h(\mathcal{D}^\bot,\mathcal{D}^\bot)\|^{2}
+\|h(\mathcal{D}^\theta,\mathcal{D}^{\theta})\|^{2}\\
&+2\bigg\{\|h(\mathcal{D}^T,\mathcal{D}^\bot)\|^{2}+\|h(\mathcal{D}^T,\mathcal{D}^{\theta})\|^{2}+\|h(\mathcal{D}^\bot,\mathcal{D}^{\theta})\|^{2}\bigg\}.
\end{align*}
In view of decomposition (\ref{g2}) and by (\ref{e12})$\sim$(\ref{e17}), which can be explicitly written as follows:
\begin{align}
\|h\|^{2}=&\displaystyle\sum^{n}_{a,b,c=1}\!\!\!g(h(\tilde{e}_{a},\tilde{e}_{b}),J\tilde{e}_{c})^{2}+
\displaystyle\sum^{n}_{a,b=1}\displaystyle\sum^{m}_{r=1}g(h(\tilde{e}_{a},\tilde{e}_{b}),e_{r}^{*})^{2}\nonumber\\
&+\displaystyle\sum^{m}_{r,s=1}\displaystyle\sum^{n}_{a=1}g(h(\bar{e}_{r},\bar{e}_{s}),J\tilde{e}_{a})^{2}+
\displaystyle\sum^{m}_{r,s,q=1}g(h(\bar{e}_{r},\bar{e}_{s}),e_{q}^{*})^{2}\nonumber\\
&+2\displaystyle\sum^{k}_{i=1}\displaystyle\sum^{n}_{a,b=1}g(h(e_{i},\tilde{e}_{a}),J\tilde{e}_{b})^{2}
+2\displaystyle\sum^{k}_{i=1}\displaystyle\sum^{m}_{r,s=1}g(h(e_{i},\bar{e}_{r}),e_{s}^{*})^{2}\label{e23}\\
&+2\!\displaystyle\sum^{n}_{a,b,c=1}g(h(\tilde{e}_{a},\tilde{e}_{b}),J\tilde{e}_{c})^{2}+2\!\displaystyle\sum^{n}_{a,b=1}
\displaystyle\sum^{m}_{r=1}g(h(\tilde{e}_{a},\tilde{e}_{b}),e_{r}^{*})^{2}\nonumber\\
&+\displaystyle\sum^{k+n+m}_{A,B=1}\displaystyle\sum^{l}_{t=1}g(h(e_{A},e_{B}),\hat{e}_{t})^{2}.\nonumber
\end{align}
Where the set $\{e_{A}\}_{1\leq A\leq (k+n+m)}$ is an orthonormal basis of $M.$ Hence, we get
\begin{equation}\label{e23*}
\begin{array}{c}
\|h\|^{2}\geq2\bigg\{\displaystyle\sum^{k}_{i=1}\displaystyle\sum^{n}_{a,b=1}g(h(e_{i},\tilde{e}_{a}),J\tilde{e}_{b})^{2}
+\displaystyle\sum^{k}_{i=1}\displaystyle\sum^{m}_{r,s=1}g(h(e_{i},\bar{e}_{r}),e_{s}^{*})^{2}\bigg\}
\end{array}.
\end{equation}
Using (\ref{e13}) and Remark (\ref{rem2}), we arrive
\begin{equation} \nonumber
\begin{array}{c}
\|h\|^{2}\geq2\bigg\{\displaystyle\sum^{k}_{i=1}\displaystyle\sum^{n}_{a,b=1}(Je_{i}(\ln f)g(\tilde{e}_{a},\tilde{e}_{b}))^{2}
+\displaystyle\sum^{k}_{i=1}\displaystyle\sum^{m}_{r,s=1}g(h(e_{i},\bar{e}_{r}),\csc\!\theta F\bar{e}_{s}))^{2}\bigg\}
\end{array}.
\end{equation}
from the inequality (\ref{e23*}). Using (\ref{e17}) and after some calculation we find
\begin{equation} \label{e23**}
\begin{array}{c}
\|h\|^{2}\geq2n\|\nabla(\ln\!f)\|^{2}+2\displaystyle\sum^{k}_{i=1}\displaystyle\sum^{m}_{r,s=1}(\csc^{2}\!\theta)
\bigg\{(Je_{i}(\ln \sigma)g(\bar{e}_{r},\bar{e}_{s}))^{2}\\
\qquad\qquad\qquad+2Je_{i}(\ln \sigma)g(\bar{e}_{r},\bar{e}_{s})e_{i}(\ln \sigma)g(\bar{e}_{r},T\bar{e}_{s})+(e_{i}(\ln \sigma)g(\bar{e}_{r},T\bar{e}_{s}))^{2}\bigg\}.
\end{array}
\end{equation}
from the last inequality. Here,
\begin{align*}
&\displaystyle\sum^{k}_{i=1}\displaystyle\sum^{m}_{r,s=1}
\bigg\{(Je_{i}(\ln \sigma)g(\bar{e}_{r},\bar{e}_{s}))^{2}Je_{i}(\ln \sigma)g(\bar{e}_{r},\bar{e}_{s})e_{i}(\ln \sigma)g(\bar{e}_{r},T\bar{e}_{s})\\
&=\displaystyle\sum^{k}_{i=1}\displaystyle\sum^{m}_{r,s=1}g(\nabla\ln \sigma,Je_{i})g(\nabla\ln \sigma,e_{i})g(\bar{e}_{r},\bar{e}_{s})g(\bar{e}_{r},T\bar{e}_{s})\\
&=-\displaystyle\sum^{m}_{r,s=1}\bigg\{\displaystyle\sum^{k}_{i=1}g(g(\nabla(\ln \sigma),e_{i})e_{i},J\nabla(\ln \sigma))\bigg\}g(\bar{e}_{r},\bar{e}_{s})g(\bar{e}_{r},T\bar{e}_{s})\\
&=-g(\nabla(\ln \sigma),J\nabla(\ln \sigma))\displaystyle\sum^{m}_{r,s=1}g(\bar{e}_{r},\bar{e}_{s})g(\bar{e}_{r},T\bar{e}_{s})=0.
\end{align*}
Thus, by Remark (\ref{rem2}), the equation (\ref{e9}) and the last yield, we deduce the inequality (\ref{e21}) from the inequality (\ref{e23**}).\\

Next, from (\ref{e23}) we see that the equality case of (\ref{e21}) holds identically if and only if the following conditions hold.
\begin{equation}
\label{e24}
\begin{array}{c}
h(\mathcal{D}^T,\mathcal{D}^T)=\{0\}, \qquad
h(\mathcal{D}^\bot,\mathcal{D}^\bot)=\{0\}, \qquad
h(\mathcal{D}^\theta,\mathcal{D}^\theta)=\{0\}
\end{array}
\end{equation}
and
\begin{equation}
\label{e25}
\begin{array}{c}
h(\mathcal{D}^\bot,\mathcal{D}^\theta)=\{0\}.
\end{array}
\end{equation}
Since $M_{T}$ is totally geodesic in $M$, from the first condition in (\ref{e24}) it
follows that $M_{T}$ is also totally geodesic in $\bar{M}.$ So, assertion \textbf{a)} follows.
Now, let $h^{\bot}$ denote the second fundamental of $M_{\bot}$ in $M$. We know that $h^{\bot}(\mathcal{D}^\bot,\mathcal{D}^\bot)\subseteq\mathcal{D}^T$
from \cite{No}. Then for  $V\in\mathcal{D}^T$ and $X,Y\in\mathcal{D}^\bot,$ we have $g(h^{\bot}(X,Y),V)=g(\nabla_{X}Y,V).$
Here, we know $\nabla_{X}Y=^{\bot}\!\nabla_{X}Y-g(X,Y)\nabla(\ln f)$ from (\ref{a3}),
where $^{\bot}\nabla$ is an induced connection on $M_{\bot}.$ Hence, we obtain
\begin{equation} \nonumber
\begin{array}{c}
g(h^{\bot}(X,Y),V)=-V(\ln f)g(X,Y)=-g(g(X,Y)\nabla(\ln f),V)
\end{array}.
\end{equation}
It follows that
\begin{equation} \label{e26}
\begin{array}{c}
h^{\bot}(X,Y)=-g(X,Y)\nabla(\ln f)
\end{array}
\end{equation}
from the last equation. Thus, combining the second condition in (\ref{e24}) and (\ref{e26}), we can deduce that
$M_{\bot}$ is a totally umbilic submanifold in $\bar{M}$
with its mean curvature vector field $-\nabla(\ln f).$ By a similar argument, we can find
$M_{\theta}$ as a  totally umbilic submanifold in $\bar{M}$ with its mean curvature vector field $-\nabla(\ln \sigma).$
So, assertion \textbf{b)} is obtained. Assertions \textbf{c)} and \textbf{d)}
immediately follow from (\ref{e24}) and (\ref{e25}), respectively.
\end{proof}
\begin{remark}  In case $\mathcal{D}^\theta=\{0\}$, Theorem \ref{thm1} coincides with Theorem 5.1 of \cite{Che2}.
In other words, Theorem \ref{thm1} is a generalization of Theorem 4.2 of \cite{Che2}.
Moreover,  Theorem \ref{thm1} coincides with Theorem 5.2 of \cite{S5} if $\mathcal{D}^\bot=\{0\}.$
Thus, Theorem \ref{thm1} is also a generalization of Theorem 5.2 of \cite{S5}.
\end{remark}

\bibliographystyle{amsplain}

\end{document}